\documentclass[12pt,leqno]{amsart}
\usepackage{graphicx}
\usepackage{soul}
\usepackage{amsmath,amssymb,amsfonts,mathrsfs, centernot}
\usepackage{amsthm}
\usepackage{color}
\usepackage{enumerate}
\usepackage{xfrac}
\usepackage{hyperref}
\usepackage{comment}

\makeatletter
\def\paragraph{\@startsection{paragraph}{4}%
  \z@\z@{-\fontdimen2\font}%
  {\normalfont\bfseries}}
\makeatother

\usepackage{tikz}
\usetikzlibrary{matrix,arrows,calc,fit,cd,positioning,intersections,arrows.meta, 3d,decorations.pathmorphing}

\newtheorem{introthm}{Theorem}

\newtheorem{theorem}{Theorem}[section]

\newtheorem{lemma}[theorem]{Lemma}
\newtheorem{proposition}[theorem]{Proposition}

\theoremstyle{definition}

\newcommand{\N}{\mathbb{N}}

\newcommand{\R}{\mathbb{R}}

\begin{document}
\pagebreak


\title{Characterization of higher rank via affine maps}

\author{David Lenze}

\address
  {Karlsruher Institut f\"ur Technologie\\ Fakult\"at f\"ur Mathematik \\
Englerstr. 2 \\
76131 Karlsruhe,
Germany}
\email{david.lenze@kit.edu}

\begin{abstract} We show that Hadamard spaces with geometric group actions admit affine maps that are not dilations, if and only if they are Riemannian symmetric spaces of higher rank, Euclidean buildings of higher rank, or split as non-trivial metric products.  \end{abstract}
\maketitle

\renewcommand{\theequation}{\arabic{section}.\arabic{equation}}
\pagenumbering{arabic}

\section{Introduction}
\subsection{Main result}

A map $f\colon X \to Y$ between metric spaces is \textit{affine} if it maps geodesics to linearly reparametrized geodesics, with a reparametrization factor that may depend on the geodesic. An affine map is called a \textit{dilation} if the reparametrization factors are independent of the geodesics. 
We call a metric space $X$ \textit{affinely rigid} if all affine maps $f\colon X \to Y$, into any metric space $Y$, are dilations.

We call a locally compact and geodesically complete CAT$(0)$ space a \textit{Hadamard space}. A \textit{geometric group action} is a
cocompact, properly discontinuous action by isometries.
In this paper, we prove:

\begin{introthm}\label{main}
    A Hadamard space admitting a geometric group action is not affinely rigid if and only if it is a Riemannian symmetric space of higher rank, a Euclidean building of higher rank or splits as a non-trivial metric product.
\end{introthm}

This result is situated within the context of \textbf{Ballmann's Higher Rank Rigidity Conjecture}, which asserts that a Hadamard space of rank at least two and admitting a geometric group action, is isometric to a Riemannian symmetric space, a Euclidean building, or a non-trivial metric product. Here a CAT$(0)$ space is said to have \textit{rank at least~$n$} if every geodesic is contained within an $n$-dimensional flat. 
As a consequence of Theorem~\ref{main}, Ballmann's conjecture can equivalently be stated as follows: \textit{any Hadamard space of rank at least two with a geometric group action admits an affine map which is not a dilation.}

We refer to Section 6 at the end of the paper for a brief discussion, exploring the necessity of the various assumptions in Theorem~\ref{main}.

Affine maps on Hadamard spaces $X$ with a geometric group action were investigated by Bennett, Mooney, and Spatzier in \cite{MR3451453}. Recall that by the generalized de Rham decomposition theorem for metric spaces \cite{MR2399098} (see also Caprace and Monod \cite{MR2574740}), $X$ uniquely decomposes as $X_1 \times \cdots \times X_n \times \mathbb{E}^d$, where each $X_i$ is irreducible and not isometric to the real line or a point. In this context, Bennett, Mooney, and Spatzier proved that any affine map $f\colon X \to Y$, where $Y$ is a CAT$(0)$ space, decomposes as a product of dilations on the $X_i$ factors and a standard affine map on the Euclidean factor $\mathbb{E}^d$. This result is analogous to Vilms' classical theorem \cite{MR262984} for affine maps between complete Riemannian manifolds.

Building on prior work of Ohta in \cite{MR1981876}, Lytchak in \cite{affineimages} completely characterized affine maps $f\colon M\to Y$ from complete Riemannian manifolds into arbitrary metric spaces $Y$. As a corollary of this characterization, every complete, simply connected Riemannian manifold admits an affine map which is not a dilation if and only if it is isometric to a Riemannian symmetric space of higher rank or a non-trivial Riemannian product. 
This result strongly resembles Theorem~\ref{main}, just as Vilms' result is a Riemannian analogue to the result of Bennett, Mooney, and Spatzier mentioned above.

The study of affine maps is motivated by their natural appearance in the investigation of various rigidity phenomena: they arise in super-rigidity, as described in the introduction in \cite{MR1981876}; in the analysis of product decompositions of metric spaces \cite{MR2399098}; and in the investigation of isometric actions on non-positively curved spaces as seen in \cite{MR1645958}. Furthermore, affine maps have recently played a crucial role in the author's work on the isometric rigidity of $L^2$-function spaces with manifold targets (cf. \cite{isom}).

Another example of this rigidity theme are the geometric constraints imposed by the existence of non-constant real-valued affine maps: Innami in \cite{MR681608} demonstrated that a complete Riemannian manifold $M$ admitting a non-constant affine map $f\colon M \to \R$ is isometric to $N\times \R$. This result was extended to geodesically complete CAT$(0)$ spaces by Alexander and Bishop in \cite{MR2224586}. Even without geodesic completeness, Lytchak and Schr\"oder  showed in \cite{MR2262730} that a CAT$(0)$ space admitting a non-constant affine map into $\R$ is still isometrically embedded into a product of another CAT$(0)$ space and a Hilbert space. Similar results for Alexandrov spaces with lower curvature bounds were obtained by Lange and Stadler in \cite{MR3803798}. Finally, in an even more general context, Lytchak and Schwer (cf. \cite{HL}) established that the existence of sufficiently many affine maps from a geodesic metric space into $\R$ implies that the space is isometric to a convex subset of a normed vector space with strictly convex norm.

\subsection{General Strategy}

Let $X$ be a Hadamard space, a subset $V \subset \partial X$ is \textit{symmetric} if for all $\xi, \xi' \in \partial X$, with $d_T(\xi, \xi') \geq \pi$, we have that $\xi \in V$ if and only if $\xi' \in V$, where $d_T$ denotes the Tits metric.

The proof of Theorem~\ref{main} relies on Stadler's recent characterization \cite{Stadler} of Riemannian symmetric spaces of higher rank, Euclidean buildings of higher rank, and non-trivial metric products via closed symmetric subsets at infinity.  

\begin{theorem}(Stadler, \cite[Theorem 1.1]{Stadler})\label{St}
	Let $X$ be a Hadamard space admitting a geometric group action. Suppose that
$\partial X$ contains a non-trivial closed symmetric subset. Then $X$ is isometric to a Riemannian symmetric space of higher rank, a Euclidean building of higher rank, or splits as a non-trivial product.
\end{theorem}

The idea for the proof of the forward direction of Theorem~\ref{main} is then to show that an affine map that is not a dilation induces a non-trivial closed symmetric subset of $\partial X$. 
Indeed, we first prove that an affine map $f\colon X \to Y$ induces a well-defined continuous function $\varphi\colon \partial X \to \mathbb{R}_{\geq 0}$, reflecting the reparametrization factors under $f$ of geodesic rays in $X$, and that this $\varphi$ is constant on minimal closed symmetric subsets generated by any $\xi \in \partial X$. 
On the other hand, we also observe that $\varphi$ is constant on the entire boundary if and only if $f$ is a dilation. 
Consequently, if $f$ is not a dilation, $\varphi$ is not constant, implying the existence of a non-trivial, closed symmetric subset. 
Theorem~\ref{St} is then invoked to complete the argument. 

Part of the technical difficulty behind this lies in showing that affine maps from Hadamard spaces into arbitrary metric spaces are Lipschitz continuous, a fact which underlies the above reasoning. 

This is addressed right at the start of our argument, in Section~\ref{sec_cont}: we begin in Lemma~\ref{cont}, by showing that continuity at a single point already implies global Lipschitz continuity of affine maps on Hadamard spaces.
Utilizing the local Lipschitz manifold structure of Hadamard spaces around \textit{regular points} (Lytchak and Nagano \cite{LN}), and building on ideas of Ohta \cite{MR1981876} and Lytchak \cite{affineimages} for the continuity of affine maps from Riemannian manifolds into metric spaces, we establish continuity at these regular points. By the above, this yields global Lipschitz continuity. 

For the converse direction of the proof of Theorem~\ref{main}, we draw upon ideas of Bennett and Schwer in \cite{P} and show, in Section~\ref{sec_Euc}, that Euclidean buildings of rank at least two are not affinely rigid, paralleling the Riemannian case mentioned above \cite{affineimages}. Combined with the fact that non-trivial metric products are also not affinely rigid (Lemma~\ref{products}), we obtain the desired result.

\section{Preliminaries}

We refer to \cite{BH, MR4701879, Stadler} for more details and proofs.
\subsection{Metric spaces} We begin by establishing notations and recalling a set of fundamental concepts and results. 

Euclidean $n$-space will be denoted by $\mathbb E^n$.  
Let $(X,d)$ be a metric space. Let $x\in X$ and $A\subset X$, then we define the distance between $x$ and $A$ as $d(x,A) :=\inf_{a\in A}d(x,a)$. For $r>0$, we define the \textit{open $r$-neighbourhood} around $A$ as $B(A,r) :=\{x\in X \vert d(x,A)<r\}$, and likewise define the \textit{closed $r$-neighbourhood} around $A$, denoted  $ B(A,r)$. We write $B(x,r):=B(\{x\},r)$, and $\overline B(x,r):= \overline B(\{x\},r)$. 

A \textit{geodesic} in $X$ is an isometric embedding of an interval into $X$. It is called a \textit{geodesic segment} if it is compact. A geodesic $\gamma\colon [0,\infty)\to X$ is called a \textit{geodesic ray}, and a geodesic $\gamma\colon \R \to X$ is a \textit{geodesic line}.

A \textit{geodesic triangle} $\Delta$ is the union of three geodesic segments connecting three points $x,y,z\in X$. A \textit{Euclidean comparison triangle} or simply \textit{comparison triangle} $\overline \Delta =\overline \Delta (x,y,z)$ of $\Delta$ is a triangle in $\mathbb E^2$ with vertices $\overline x, \overline y, \overline z \in \mathbb E^2$ such that $d(x,y)=|\overline x-\overline y|$, $d(x,z)=|\overline x-\overline z|$ and $d(y,z)=|\overline y-\overline z|$, and is unique up to an isometry of $\mathbb E^2$. 

The metric space $X$ is a \textit{geodesic metric space} if every pair of points in $X$ is joined by a geodesic. Finally $X$ is \textit{geodesically complete} if every geodesic segment is contained in a complete local geodesic.

\paragraph{Affine maps}
A map $f\colon X \rightarrow Y$ between metric spaces $X$ and $Y$ is called \textit{affine} if it maps geodesics to linearly reparametrized geodesics.  In other words, for any geodesic $\gamma$ in $X$, there exists a constant $\rho(\gamma)\geq 0$, called the \textit{reparametrization factor}, such that $d_Y(f(\gamma(t)),f(\gamma(t')))= \rho(\gamma) |t-t'|$ for all $t,t' \in I$. We stress that $\rho(\gamma)$ may depend on $\gamma$; if $\rho(\gamma)$ is independent of $\gamma$, then $f$ is a \textit{dilation}. We call a geodesic metric space $X$ \textit{affinely rigid} if every affine map $f\colon X\to Y$, into any metric space $Y$, is a dilation. 

By \cite[I.5.3]{BH}, the components of a geodesic in a metric product are linearly reparametrized geodesics, showing that projection maps $X\times Y \to X$ are affine. Specifically, projections from non-trivial metric products onto one of their factors serve as examples of affine maps that are not dilations, and thus we have: \begin{lemma}\label{products}
	Non-trivial metric products are not affinely rigid. 
\end{lemma}

\paragraph{CAT(0) spaces}

A complete geodesic metric space is a CAT$(0)$ \textit{space} if, for any geodesic triangle, the distance between any two points on its sides is less than or equal to the distance between the corresponding points on the sides of its comparison triangle in $\mathbb E^2$. See \cite[II.1]{BH} for more details. Following Stadler in \cite{Stadler}, a locally compact, geodesically complete CAT$(0)$ space is a \textit{Hadamard space}.

Let $(X,d)$ be a CAT$(0)$ space. Two geodesic rays $\gamma,\gamma^\prime\colon [0,\infty) \to X$ are \textit{asymptotic} if $\sup_{t\geq 0}d(\gamma(t),\gamma^\prime(t))<\infty$. Being asymptotic is an equivalence relation and the set of equivalence classes of geodesic rays is the \textit{boundary at infinity} $\partial X$ of $X$. The equivalence class of a geodesic ray $\gamma$ is denoted $\gamma(\infty)$. For any $x\in X$ and $\xi\in \partial X$, there exists a \textit{unique} geodesic ray $\gamma_{x,\xi}\colon [0,\infty)\to X$ which issues from $x$ and with $\gamma_{x,\xi}(\infty)=\xi$ (cf. for example \cite[II. 8.2]{BH}).

The boundary at infinity $\partial X$ equipped with the \textit{cone topology} is the \textit{visual boundary} (cf. \cite[II.8]{BH}). Pick a base point $x\in X$. An explicit neighbourhood basis of the cone topology around any point $\xi \in \partial X$ can be given by neighbourhoods of the following form: Given $r>0$ and $\delta>0$, we define $$U(\xi, r, \delta):=\{ \xi^\prime \in \partial X \vert d(\gamma_{x,\xi}(r), \gamma_{x,\xi^\prime}(r))< \delta \}.$$ 
The \textit{Tits boundary} of $X$ is the boundary at infinity equipped with the \textit{Tits metric} $d_T$, the intrinsic metric associated to the \textit{angular metric}, see again \cite[II.8]{BH}. 

Following the notations and definitions in \cite{MR4701879}, two points $\xi, \xi^\prime \in \partial X$ are \textit{antipodal}, if $d_T(\xi, \xi^\prime)\geq \pi$, and \textit{visually antipodal}, if there exists a geodesic line $c \colon \R \to X$ connecting these two points at infinity, i.e. $c(-\infty)=\xi$ and $c(\infty)=\xi^\prime$.
The \textit{set of (visually) antipodal points} of a subset $V\subset \partial X$ is denoted by $\text{Ant}(V)$, $\text{Ant}_{vis}(V)$ respectively, and we define $\text{Ant}^{j+1}(V) = \text{Ant}(\text{Ant}^j(V))$ and $\text{Ant}_{vis}^{j+1}(V) = \text{Ant}_{vis}(\text{Ant}_{vis}^j(V))$ inductively. A subset $V\subset \partial X$ is \textit{(visually) symmetric} if $$\text{Ant}_{(vis)}(V)\subset V.$$ 
Finally we denote the \textit{minimal visual symmetric subset} of $V\subset \partial X$ by $\text{A}_{vis}(V):=\bigcup_{j\in \N} \text{Ant}_{vis}^{j}(V) $ and record the following facts: \begin{lemma}\label{St2}(Stadler, \cite[Corollaries 3.5 and 3.6]{MR4701879})
	Let $X$ be a Hadamard space and let $A\subset \partial X$. Then the following hold: 
	\begin{enumerate}
		\item If $A$ is visually symmetric, then the closure $\overline A$ is symmetric.
		\item $A$ is closed and symmetric if and only if it is closed and visually symmetric.  
	\end{enumerate}
\end{lemma}

\subsection{Euclidean buildings}
We briefly recall Euclidean buildings, using the metric approach of Kleiner and Leeb \cite{MR1608566}, which is equivalent to Tits' original definition \cite{MR843391} (see \cite{MR1796138}). We closely follow Kramer \cite{kramer}, to which we refer for further details and proofs. Note that the below definition is a generalization of the classic notion, since we do not assume the group $W$ to be discrete. As a consequence, this definition encompasses non-locally compact Euclidean buildings.

Let $W = \mathbb{R}^{m} \rtimes W_0$ be the semidirect product of the additive group $\mathbb{R}^{m}$, and a spherical Coxeter group $W_0$ acting in its natural orthogonal representation on Euclidean space $\mathbb{E}^{m}$. We call $W$ an affine Weyl group. From the reflection hyperplanes of $W_0$ we obtain a decomposition of $\mathbb{E}^{m}$ into walls, half-spaces and Weyl chambers. The $W$-translates of these in $\mathbb{E}^{m}$ are still called walls, half-spaces, and Weyl chambers.

Let $(X,d)$ be a metric space. A \textit{chart} is an isometric embedding $\varphi\colon  \mathbb{E}^{m} \longrightarrow X$. The images of charts are called \textit{apartments}. The metric space $X$ is a \textit{Euclidean building} if there exists a collection $\mathcal A$ of charts such that the following five properties hold:

\begin{enumerate}
    \item[(A1)] For all $\varphi \in \mathcal{A}$ and $w \in W$, the composition $\varphi \circ w$ is in $\mathcal{A}$.
    \item[(A2)] The charts are $W$-compatible (i.e., for any two charts $\varphi, \psi \in \mathcal{A}$, $\varphi^{-1}(\psi(\mathbb{E}^{m}))$ is convex, and there is an element $w \in W$ such that $\psi \circ w$ and $\varphi$ agree on this convex set).
    \item[(A3)] Any two points $p, q \in X$ are contained in some apartment.
\end{enumerate}
We map walls, half spaces and Weyl chambers from $\mathbb{E}^{m}$ into $X$ using our collection of charts, the first three axioms ensure that these notions will be well-defined and independent of the choice of the chart.
\begin{enumerate}
    \item[(A4)] If $C, D \subset X$ are Weyl chambers, there is an apartment $A$ such that the intersections $A \cap C$ and $A \cap D$ contain Weyl chambers.
    \item[(A5)] For every apartment $A \subseteq X$ and every $p \in A$, there is a 1-Lipschitz retraction $h\colon  X \longrightarrow A$ with $h^{-1}(p) = \{p\}$.
\end{enumerate}
By \cite[II 10A.4]{BH}, Euclidean buildings are CAT$(0)$ spaces.  

\section{Continuity of affine maps}\label{sec_cont}

In this section, we prove that an affine map on a Hadamard space is always Lipschitz continuous. 

We begin by reducing global Lipschitz continuity to continuity at a single point in the Hadamard space:
 
 \begin{lemma}\label{cont}
    Let $f\colon  X \to Y$ be an affine map from a geodesically complete CAT$(0)$ space $X$ to any metric space $Y$. If $f$ is continuous at a single point in $X$, then $f$ is globally Lipschitz continuous.
\end{lemma}

\begin{proof}
Let $f$ be continuous at $x\in X$, then $\sup_{\xi\in \partial X} \rho(\gamma_{x,\xi})<\infty$, where $\rho(\gamma_{x,\xi})$ denotes the reparametrization factor under the affine map $f$ of the unique geodesic ray from $x$ to $\xi$. Indeed, otherwise there would exist a sequence $(\gamma_n)_{n\in \N}$ of geodesic rays issuing from $x$ with $\rho(\gamma_n)\geq n$. In that case $\gamma_n(\frac{1}{n})\to x$, but also $d_Y(f(x),f(\gamma_n(\frac{1}{n})))\geq 1$, a contradiction. 

Denote $\lambda:= \sup_{\xi\in \partial X} \rho(\gamma_{x,\xi})$. Since by geodesic completeness, the unique geodesic from $x$ to any $y\in X$ can be extended to a geodesic line, we obtain a local Lipschitz bound: $d_Y(f(x),f(y))\leq \lambda \cdot d_X(x,y).$

Now assume for contradiction that $f$ is not globally $\lambda$-Lipschitz continuous. Then there exist $y\in X$ and a geodesic line $\gamma\colon \R \to X$ such that $\gamma(0)=y$ and $\rho :=\rho(\gamma)> \lambda$.
 But then for any $t\geq 0$, by applying the triangle inequality twice, we obtain: \begin{align*}
	d_Y(f(y),f(\gamma(t)))&\leq d_Y(f(y),f(x))+d_Y(f(x),f(\gamma(t))) \\&\leq \lambda(d_X(y,x)+d_X(x,\gamma(t))) \\ &\leq \lambda(2d_X(y,x)+d_X(y,\gamma(t)))\\&=\lambda(2d_X(y,x)+t).
\end{align*}

Therefore, $(\rho-\lambda)\cdot t\leq 2\lambda d_X(y,x)<\infty$. The left hand side tends to infinity, while the right hand side remains bounded, as $t\to \infty$. This is a contradiction and proves that $f$ is Lipschitz continuous. 
\end{proof}

Let $v,w \in \mathbb{E}^n$, we denote by $l_{v,w}:[0,|v-w|]\to \mathbb E^n$ the line segment from $v$ to $w$ parametrized by arclength. 
The following auxiliary result establishes a uniform control on the deviation of $\lambda$-bi-Lipschitz curves from straight line segments connecting designated endpoints, given uniform bounds on the constants $\lambda$ and the distances of the endpoints of the curves to the designated endpoints of the line segments.

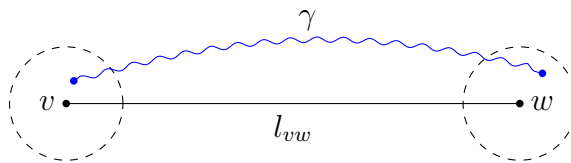
\begin{figure}[h]
 \begin{center}
  \begin{tikzpicture}
   \coordinate (A) at (0,0);
   \coordinate (B) at (6,0);

   \draw[dashed] (A) circle (0.75);
   \draw[dashed] (B) circle (0.75);
   \fill (A) circle (0.05);
   \fill (B) circle (0.05);

   \draw (A) -- (B) node[pos=0.5, below] {$l_{vw}$};
   \draw[blue, decorate, decoration={snake, amplitude=1}] (0.1,0.3) to[bend left=15] node[pos=0.5, above, black] {$\gamma$} (6.3,0.4);
   \fill[blue] (0.1,0.3) circle (0.05);
   \fill[blue] (6.3,0.4) circle (0.05);

   \node[left] at (A) {$v$};
   \node[right] at (B) {$w$};
  \end{tikzpicture}
 \end{center}
 \caption{A $\lambda(\delta)$-bi-Lipschitz curve with endpoints contained in $\eta(\delta)$-balls centred at $v$ and $w$ respectively.}
\end{figure}

\begin{lemma}\label{lem}
Suppose that for any $\delta>0$, there exist $\lambda(\delta)\geq 1$ and $\eta(\delta)>0$ with $\lambda(\delta) \to 1$ and $\eta(\delta) \to 0$ as $\delta \to 0$. Then for any $D>0$, there exists $\beta(\delta)>0$,  with $\beta(\delta)\to 0$ as $\delta \to 0$, and such that for any $v, w \in \mathbb{E}^n$, $|v-w|\leq D$, and any $\lambda(\delta)$-bi-Lipschitz curve $\gamma: [0, |v-w|] \to \mathbb{E}^n$ with $|\gamma(0) - v|, |\gamma(|v-w|) - w| < \eta(\delta)$, we have $\|\gamma - l_{v,w}\|_{\infty} \leq \beta(\delta)$.
\end{lemma}

\begin{proof}
	Let $v, w \in \mathbb{E}^n$. Define $\mathcal{C}(v, w, \delta)$ as the set of $\lambda(\delta)$-bi-Lipschitz curves $\gamma: [0, |v-w|] \to \mathbb{E}^n$ satisfying $|\gamma(0) - v|, |\gamma(|v-w|) - w| < \eta(\delta)$. Let $\beta(v, w, \delta) = \sup_{\gamma \in \mathcal{C}(v, w, \delta)} \|\gamma - l_{v,w}\|_{\infty}$. 
	
	First note that for symmetry reasons, if $|v-w|=|v^\prime-w^\prime|$, then $\beta(v, w, \delta)=\beta(v^\prime, w^\prime, \delta)$. Thus, $\beta(v, w, \delta)$ depends only on $d:=|v-w|$, and we write $\beta(d,\delta)$.

Let $d^\prime \geq d$, and let $e$ be a standard unit vector. For $\gamma \in \mathcal C(0,d\cdot e,\delta)$, extend $\gamma$ to $\widetilde\gamma:[0,d^\prime] \to \mathbb E^n$ by setting $\widetilde \gamma(t)=\gamma(t)$ for $t\leq d$, and $\widetilde \gamma (t)=t\cdot e$ for $t\geq d$. Then $\widetilde \gamma \in \mathcal C(0,d^\prime \cdot e,\delta)$ and $\|\widetilde \gamma -l_{0,d^\prime \cdot e}\|_\infty=\|\gamma-l_{0,d\cdot e}\|_\infty$. Therefore, $\beta(d,\delta)\leq \beta(d^\prime,\delta)$. Thus, we set $\beta(\delta):=\beta(D,\delta)$. 

Finally, suppose $\beta(\delta)\nrightarrow 0$ as $\delta\to 0$. Then there exists $\epsilon>0$ and $\delta_k\to 0$ such that $\beta(\delta_k)>\epsilon$. Thus, there exist $\lambda(\delta_k)$-bi-Lipschitz curves $\gamma_k\in\mathcal{C}(0,D\cdot e,\delta_k)$ with $\|\gamma_k - l_{0,D\cdot e}\|_\infty\geq\epsilon$. By Arzel\`a--Ascoli, a subsequence of $(\gamma_k)_{k\in \N}$ converges to some $\gamma$, and since $\lambda(\delta_k)\to 1$ and $\eta(\delta_k)\to 0$, we deduce that $\gamma=l_{0,D\cdot e}$. This contradicts $\|\gamma - l_{0,D\cdot e}\|_\infty=\lim_{k\to \infty}\|\gamma_k - l_{0,D\cdot e}\|_\infty \geq\epsilon$, proving $\lim_{\delta \to 0}\beta(\delta)= 0$.\end{proof}

\begin{proposition}\label{affine_contin}
	Let $X$ be Hadamard space. Any affine map $f\colon X\to Y$ is Lipschitz continuous.  
\end{proposition}

The proof adapts to CAT$(0)$ spaces ideas used by Ohta and Lytchak to prove the continuity of affine maps from Riemannian manifolds to metric spaces, see \cite{MR1981876} and \cite{affineimages}.

\begin{proof}[Proof of Proposition~\ref{affine_contin}]
By Lemma~\ref{cont}, it suffices to show that $f$ is continuous at a single point. 
By Lytchak and Nagano (cf. Theorem 1.2, Section 14.2 and 14.3 in \cite{LN}), there exists a \textit{regular point} $x \in X$ such that the tangent cone $T_xX$ is isometric to $\mathbb E^n$ for some $n \in \mathbb{N}$. Moreover, for every $\delta > 0$, there exists $r(\delta) > 0$ and a $(1+\delta)$-bi-Lipschitz embedding $F_\delta\colon  B_{r(\delta)}(0) \to X$ with open image and $F_\delta(0) = x$. We show that $f$ is continuous at this regular point $x$. 

Let $v_1,...,v_{n+1}\in \mathbb E^n$ be the unit vertices of a standard $n$-simplex $\Delta=\text{Conv}(v_1, \dots, v_{n+1})\subset \mathbb E^n$ centred at the origin. For $0 \le k \le n$ the $k$-skeleton $\Delta_k$ is the union of all $k$-dimensional faces, $F_A$, where $A \subset \{1, \dots, n+1\}$, $|A| = k+1$ and $F_A := \text{Conv}(\{v_i \mid i \in A\})$. In other words,
$\Delta_k = \bigcup_{|A| = k+1} F_A.$
Observe that $\Delta_0 = \{v_1, \dots, v_{n+1}\}$, $\Delta_1$ is the union of one dimensional edges, $\Delta_{n-1}=\partial \Delta$, and $\Delta_n = \Delta$.

Take any $\delta\in (0,1)$. Since the image of $F_\delta$ is open, there exists $\rho:=\rho(\delta)>0$ such that $F_\delta(\rho v_i)$, $i=1,...,n+1$ are all contained in a ball $B(x,\rho^\prime)\subset \text{im} F_\delta$ of radius $\rho^\prime:=\rho^\prime(\delta)< 2\rho$ centred at $x$.

Let $g_{0,\delta}:\Delta_0 \to X$ be defined by $g_{0,\delta}(v_i)=F_\delta(\rho v_i)\in B(x,\rho^\prime)$ for $i=1,...,n+1$. Inductively define successive extensions $g_{k,\delta}:\Delta_k \to X$ as follows: Given $g_{k-1,\delta}$, for $A \subset \{1,...,n+1\}$, $|A|=k+1$, and $w\in F_{A\setminus \{v\}}$, let $\gamma$ be the geodesic from $g_{k-1,\delta}(v)$ to $g_{k-1,\delta}(w)$, where $v:=v_{m(A)}$ and $m(A) = \min(A)$. Then, for $t\in [0,1]$, we define $$g_{k,\delta}((1-t)v+tw) = \gamma(t d(g_{k-1,\delta}(v), g_{k-1,\delta}(w))).$$

This completely characterizes $g_{k,\delta}$ in a well-defined way. Indeed, on $\Delta_{k-1}$, it is simply $g_{k-1,\delta}$. On $\Delta_k \setminus \Delta_{k-1}$ on the other hand, take any point $u \in \Delta_k \setminus \Delta_{k-1}$. This point lies in exactly one $k$-face $F_A$, where $A \subset \{1,...,n+1\}$, and there exists a unique $w \in F_{A \setminus \{v\}}$, $v:=v_{m(A)}$, such that the line connecting $v$ to $u$ intersects the face $F_{A \setminus \{v\}}$ at $w$. Furthermore, there is a unique $t \in (0, 1)$ such that $u = (1-t)v + tw$. Thus $g_{k,\delta}:\Delta_k \to X$ is a well-defined extension of $g_{k-1,\delta}:\Delta_{k-1} \to X$.

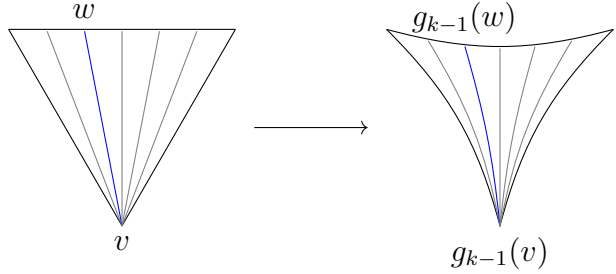
\begin{figure}
 \begin{center}
  \begin{tikzpicture}
   \coordinate (A) at (-1.5,0);
   \coordinate (B) at (1.5,0);
   \coordinate (C) at (0,-{1.5*sqrt(3)});
   \draw (B) -- (C) -- (A);
   \draw (A) --
    node[pos={1/6}, inner sep=0] (w1) {}
    node[pos={2/6}, inner sep=0] (w2) {}
    node[pos={3/6}, inner sep=0] (w3) {}
    node[pos={4/6}, inner sep=0] (w4) {}
    node[pos={5/6}, inner sep=0] (w5) {}
    (B);

   \draw[gray] (C) -- (w1);
   \draw[blue] (C) -- (w2);
   \draw[gray] (C) -- (w3);
   \draw[gray] (C) -- (w4);
   \draw[gray] (C) -- (w5);

   \node[above] at (w2) {$w$};
   \node[below] at (C) {$v$};

   \coordinate (A') at (3.5,0);
   \coordinate (B') at (6.5,0);
   \coordinate (C') at (5,-{1.5*sqrt(3)});
   \draw (B') to[bend right=15] (C') to[bend right=15] (A');
   \draw (A') to[bend right=15]
    node[pos={1/6}, inner sep=0] (w'1) {}
    node[pos={2/6}, inner sep=0] (w'2) {}
    node[pos={3/6}, inner sep=0] (w'3) {}
    node[pos={4/6}, inner sep=0] (w'4) {}
    node[pos={5/6}, inner sep=0] (w'5) {}
    (B');

   \draw[gray] (C') to[bend right=10] (w'1);
   \draw[blue] (C') to[bend right=5] (w'2);
   \draw[gray] (C') -- (w'3);
   \draw[gray] (C') to[bend left=5] (w'4);
   \draw[gray] (C') to[bend left=10] (w'5);

   \node[above] at (w'2) {$g_{k-1}(w)$};
   \node[below] at (C') {$g_{k-1}(v)$};

   \draw[->] (1.75,-{0.75*sqrt(3)}) -- (3.25,-{0.75*sqrt(3)});
  \end{tikzpicture}
 \end{center}
 \caption{Inductive construction  of $g_{k}$ from $g_{k-1}$.}
\end{figure}

Note that by the convexity of balls in CAT$(0)$ spaces, $\text{im}(g_{k,\delta})\subset B(x,\rho^\prime)$. 
Next define $h_{k,\delta}:\Delta_k\to \mathbb E^n$ by $h_{k,\delta}(v):=\frac{1}{\rho}F_\delta^{-1}(g_{k,\delta}(v))$; and let $j_{k}:\Delta_k \to \mathbb E^n\setminus \{0\}$ be the standard inclusions. We show inductively that for any $\epsilon>0$, there exits $\delta >0$ such that $\|h_{k,\delta}-j_k\|_\infty<\epsilon.$ 

For the base case $k=0$, note that $h_{0,k}=j_0$, and so the claim holds. 

Now assume the claim holds for $k-1\geq 0$. Then there exists $\eta(\delta)\geq 0$ with $\eta(\delta) \to 0$ as $\delta \to 0$, and such that $\|h_{k-1,\delta}-j_{k-1}\|_\infty<\eta(\delta).$ Take a $k$-dimensional face $F_A$, $A\subset \{1,...,n+1\}$ and $|A|=k+1$, set $v:=v_{m(A)}$, and choose some $w\in F_{A\setminus \{v\}}$. We know that there exists $\mu(k)>0$ such that $\mu(k)\leq|v-w|\leq 2$.

Now we show that the line segment $l_{v,w}:[0,|v-w|]\to \mathbb E^n$ is mapped to a $\lambda(\delta)$-bi-Lipschitz curve $\widetilde \gamma$ connecting $h_{k-1,\delta}(v)=v$ to $h_{k-1,\delta}(w)$, where $\lambda(\delta)\to 0$ as $\delta\to 1$ and $\lambda(\delta)$ does not depend on the choice of $v$ and $w$, but only on $\delta$ and the dimension $k$. 

 Indeed, we have $h_{k,\delta}(l_{vw}(t))=\frac{1}{\rho}F_\delta^{-1}(\gamma(t\cdot \frac{d(g_{k-1,\delta}(v),g_{k-1,\delta}(w))}{|v-w|})):=\widetilde \gamma(t),$ where $\gamma$ is the geodesic connecting $g_{k-1,\delta}(v)$ to $g_{k-1,\delta}(w)$. By the bi-Lipschitz property of $F_\delta$, we know that $
	\frac{1}{\rho}(1+\delta)^{-1}\frac{d(g_{k-1,\delta}(v),g_{k-1,\delta}(w))}{|v-w|}|t-t^\prime| \leq  |\widetilde\gamma(t)-\widetilde\gamma(t^\prime)|\leq  \frac{1}{\rho}(1+\delta)\frac{d(g_{k-1,\delta}(v),g_{k-1,\delta}(w))}{|v-w|}|t-t^\prime|.$
At the same time, by the triangle inequality, again the bi-Lipschitz property of $F_\delta$, and the assumption that $\|h_{k-1,\delta}-j_{k-1}\|_\infty<\eta(\delta)$,  we deduce that \begin{align*}
	\frac{d(g_{k-1,\delta}(v),g_{k-1,\delta}(w))}{|v-w|} &\leq \rho (1+\delta)\frac{|h_{k-1,\delta}(v)-h_{k-1,\delta}(w)|}{|v-w|}\\&\leq \rho (1+\delta)\left(1+\frac{\eta(\delta)}{|v-w|} \right)\\&\leq \rho(1+\delta)\left(1+\frac{\eta(\delta)}{\mu(k)} \right),
\end{align*} and likewise, $\rho(1+\delta)^{-1}\left(1-\frac{\eta(\delta)}{\mu(k)} \right)\leq \frac{d(g_{k-1,\delta}(v),g_{k-1,\delta}(w))}{|v-w|}$. Thus for sufficiently small $\delta$, we have shown that $\widetilde \gamma$ is indeed a $\lambda(\delta)$-bi-Lipschitz curve, where $\lambda(\delta)=(1+\delta)^2\left(1-\frac{\eta(\delta)}{\mu(k)} \right)^{-1} \to 1$ as $\delta \to 0$. 

Therefore by Lemma~\ref{lem}, and since $|v-w|\leq 2$, there exists $\beta(\delta)>0$, independent of the choice of $A$, $v$ and $w$, such that $\beta(\delta)\to 0$ and, $$\|h_{k,\delta}(l_{vw}) -l_{vw}\|_\infty \leq \beta(\delta).$$ 

Since any point of $\Delta_k$ is contained in such a line segment $l_{vw}$, this shows that indeed $\|h_{k,\delta}-j_k\|_\infty \leq \beta(\delta)$. 

Now choose $\epsilon > 0$ small enough, so that $ B(\partial \Delta, \epsilon)$ avoids a neighbourhood of the origin. By the above, there exists $\delta > 0$ with $\|h_{{n-1},\delta}-j_{n-1}\|_\infty<\epsilon$, and so $\Gamma:=\text{im}(h_{n-1,\delta})\subset \mathbb E^n\setminus \{0\}$. Define $h:\partial \Delta=\Delta_{n-1} \to \Gamma$ by $h(v):=h_{n-1,\delta}(v)$, i.e. $h$ is simply $h_{n-1,\delta}$ with target restricted to the image. Since $\|\iota_1 \circ h-j\|_\infty<\epsilon$, $\iota_1\circ h$ is homotopic to $j:=j_{n-1}$, where $\iota_1:\Gamma \to \mathbb  E ^n\setminus \{0\}$ is the standard inclusion.

We now show that $\Gamma$, while avoiding the origin, must intersect every ray emanating from the origin.

To see this, assume, for contradiction, that there exists a ray $\nu\colon [0, \infty) \to \mathbb{E}^n$, emanating from the origin that does not intersect $\Gamma$. Then $\Gamma \subset \mathbb{E}^n \setminus \text{im}(\nu)$, and since $\mathbb{E}^n \setminus \text{im}(\nu)$ is contractible, $H_{n-1}(\mathbb{E}^n \setminus \text{im}(\nu)) = 0$. Consequently, $(\iota_2)_*\colon H_{n-1}(\Gamma) \to H_{n-1}(\mathbb{E}^n \setminus \text{im}(\nu))$, induced by the standard inclusion $\iota_2\colon \Gamma \hookrightarrow \mathbb{E}^n \setminus \text{im}(\nu)$, is the zero map. Let $\iota_3\colon \mathbb{E}^n \setminus \text{im}(\nu) \hookrightarrow \mathbb{E}^n \setminus \{0\}$ be the standard inclusion. Then $(\iota_1)_*=(\iota_2)_* \circ (\iota_3)_*$ is also the zero map.

However, $(\iota_1\circ h)$ is homotopic to $j\colon \partial \Delta \hookrightarrow \mathbb{E}^n \setminus \{0\}$, which is a homotopy equivalence. Thus $(\iota_1)_* \circ (h)_*=(j)_*$ is an isomorphism from $H_{n-1}(\partial \Delta)\cong \mathbb{Z}$ to $H_{n-1}(\mathbb E^n\setminus \{0\})\cong \mathbb{Z}$.
But also, since $(\iota_1)_*$ is the zero map, $(j)_*$ must also be the zero map, a contradiction. Thus, as claimed, $\Gamma$ intersects every ray.

Define $A_{k}:=\text{im}(g_{k,\delta})= F_\delta(\rho \cdot \text{im}(h_{k,\delta}))$, $k\leq n-1$. The above implies that $A:=A_{n-1}=F_\delta(\rho \cdot \Gamma)$ does not contain $x$, but intersects any geodesic ray $\gamma\colon [0,\infty) \to X$ which issues from $x$. We can use this fact to bound the map $f$ around $x$:
let $a:=\max\{d_Y(f(x),f(F_\delta(\rho v_i))\vert i=1,...,n+1\}$ and $b:=d(x,A)$. By induction, and by the triangle inequality, we know that $d_Y(f(x),f(x^\prime))\leq 2^k a $ for all $x^\prime \in A_{k}$. Since $A$ is compact, we also have that $b>0$. Therefore, for all $x^\prime\in A$, we know that $d_Y(f(x),f(x^\prime)) \leq \lambda d(x,x^\prime)$, where $\lambda:=\frac{2^{n-1} a}{b}$.

Since $f$ is affine and $A$ intersects any ray which issues from $x$, this implies that $d_Y(f(x),f(x^\prime))\leq \lambda d(x,x^\prime)$ for all $x^\prime \in X$, and therefore $f$ is continuous at $x$. This completes the proof. \end{proof}

\section{Affine maps on Euclidean buildings}\label{sec_Euc}

\begin{lemma}\label{buildings}
Euclidean buildings of higher rank are not affinely rigid. \end{lemma}

Our proof depends upon insights by Bennett--Schwer \cite{P}, on modifying the model geometry of the apartments of a Euclidean building.

\begin{proof}[Proof of Lemma~\ref{buildings}]
Let $(X,d)$ be a Euclidean building of rank at least two and let $W = \R^n \rtimes W_0$ be the affine Weyl group associated with $X$, where $W_0$ is a spherical Coxeter group. By Definition~\ref{buildings}, there exists a collection $\mathcal A$ of charts $\varphi\colon  (\R^n,\|\cdot\|_2) \to (X,d)$ satisfying assumptions (A1)-(A5). 

Now the idea is to equip $\R^n$ with a non-Euclidean norm $\|\cdot\|_B$ which is invariant under the action of $W_0$. Indeed in that case, we observe that by the definition of Euclidean buildings, if we take any two points $x,y\in X$ and some apartment $\varphi(\R^n)$, $\varphi \in \mathcal A$ containing $x$ and $y$, the expression $\|\varphi_A^{-1}(x) - \varphi_A^{-1}(y)\|_B$ is independent of the particular choice of $\varphi$. Thus we obtain the well-defined distance function $$d_B(x, y) := \|\varphi_A^{-1}(x) - \varphi_A^{-1}(y)\|_B,$$ for any $\varphi \in \mathcal A$ with $x,y\in \varphi(\R^n)$. By Bennett and Schwer \cite[Theorem 3.5]{P}, this distance function satisfies the triangle inequality and is thus a metric. 
 
Since any geodesic $\gamma\colon  I\to (X,d)$ lies in an apartment $\varphi(\R^n)$, $\varphi\in \mathcal A$, we can translate this back to a geodesic $\varphi_A^{-1}(\gamma)\colon I \to (\R^n,\|\cdot\|_2)$. This map is still a geodesic in $\R^n$ after changing to the norm $\|\cdot\|_B$. Hence, by the definition of the metric $d^\prime$, $\gamma\colon I \to (X,d^\prime)$ is also a geodesic. 

Therefore the identity $\text{id}\colon (X,d) \to (X,d^\prime)$ is an affine map. Finally, since the norm $\|\cdot\|_B$ is non-Euclidean, this affine map is not a dilation. 

Thus it only remains to construct such a non-Euclidean norm invariant under the action of $W_0$. To that end, simply choose a special \textit{$W_0$-chamber} or \textit{alcove} $C$ (see \cite[Remarks 10.33 (d)]{MR2439729}) adjacent to the origin and observe that the orbit $B = \bigcup_{w\in W_0}wC$ is a symmetric convex body invariant under $W_0$. Since $n \geq 2$, $B$ is not a Euclidean ball. Thus, $B$ induces a non-Euclidean norm $\|\cdot\|_B$ on $\mathbb{R}^n$, which is invariant under $W_0$.
\end{proof}
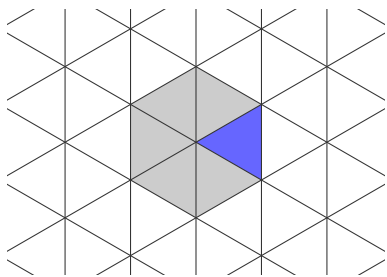
\begin{figure}[h]
	\begin{tikzpicture}[scale=1,
	extended line/.style={shorten >=-#1,shorten <=-#1}, extended line/.default=35cm,
	every node/.style={inner sep=1.8pt, circle, draw}]
	
	\clip (-2.5, -1.75) rectangle (2.5, 1.75);
	
	\fill[black!20] ($(0, 0)$) -- ($(0, -1)$) -- ($({0.5*sqrt(3)}, -0.5)$);
	\fill[black!20] ($(0, 0)$) -- ($(0, 1)$) -- ($({0.5*sqrt(3)}, 0.5)$);
	\fill[blue!60] ($(0, 0)$) -- ($({0.5*sqrt(3)}, 0.5)$) -- ($({0.5*sqrt(3)}, -0.5)$);
	\fill[black!20] ($(0, 0)$) -- ($(-{0.5*sqrt(3)}, 0.5)$) -- ($(-{0.5*sqrt(3)}, -0.5)$);
	\fill[black!20] ($(0, 0)$) -- ($(0, 1)$) -- ($(-{0.5*sqrt(3)}, 0.5)$);
	\fill[black!20] ($(0, 0)$) -- ($(0, -1)$) -- ($(-{0.5*sqrt(3)}, -0.5)$);

	\begin{scope}[black!80]
	\newcommand{\rows}{5}
	\foreach \row in {-\rows, ...,\rows} {
		\draw [extended line] ($\row*({0.5*sqrt(3)}, 0.5)$) -- ($(0,\rows)+\row*({0.5*sqrt(3)}, -0.5)$);
		\draw [extended line] ($\row*(0, 1)$) -- ($({\rows/2*sqrt(3)}, \rows/2)+\row*({-0.5*sqrt(3)}, 0.5)$);
		\draw [extended line] ($\row*(0, 1)$) -- ($({-\rows/2*sqrt(3)}, \rows/2)+\row*({0.5*sqrt(3)}, 0.5)$);
	}
	\end{scope}
	\end{tikzpicture}
	\caption{Tessellation of $\mathbb E^n$ induced by $W$ with $B$ coloured and the special alcove $C$ highlighted in blue.}
\end{figure}

\section{Proof of the main result}

\begin{proof}[Proof of Theorem~\ref{main}]

By Lemma~\ref{products} for direct products, \cite[Theorem 1.1]{affineimages} for Riemannian symmetric spaces and Lemma~\ref{buildings} for Euclidean buildings, if $X$ splits as a non-trivial direct product, is a Riemannian symmetric space of rank at least two or a Euclidean building of rank at least two, $X$ is not affinely trivial. This establishes the first direction of Theorem~\ref{main}.

For the other direction, assume that $X$ is not affinely rigid, i.e. there exists an affine map $f\colon  X \to Y$ which is not a dilation.  

First, note that by Proposition~\ref{affine_contin}, the affine map $f$ is Lipschitz continuous with Lipschitz constant $\|f\|>0$. 
We begin by showing that if $\gamma_1,\gamma_2\colon [0,\infty) \to X$ are asymptotic rays, meaning $\sup_{s\geq 0} d_X(\gamma_1(s),\gamma_2(s)) < \infty$, then $f\circ\gamma_1$ and $f\circ\gamma_2$ are reparametrized under $f$ by the same factor. Indeed, since $f$ is affine, $f\circ\gamma_i$ are linearly reparametrized geodesics with factors $\rho_i$, $i=1,2$ respectively. Suppose now $\rho_1 \neq \rho_2$. We can assume that $\rho_1<\rho_2$. Now for any $t \geq 0$, using both the triangle inequality and the Lipschitz property of $f$, we have
\begin{align*}
(\rho_2-\rho_1) \cdot t &\leq d_Y(f(\gamma_1(0)),f(\gamma_2(0))) + d_Y(f(\gamma_1(t)),f(\gamma_2(t)))\\
&\leq 2\|f\| \sup_{s\geq 0} d_X(\gamma_1(s),\gamma_2(s)) < \infty.
\end{align*}
Letting $t \to \infty$ yields a contradiction, so $\rho_1 = \rho_2$.

For $x\in X$ and $\xi\in \partial X$, let $\gamma_{x,\xi}\colon [0,\infty) \to X$ be the unique ray from $x$ to $\xi$. Define $\varphi_x(\xi) = d_Y(f(\gamma_{x,\xi}(1)),f(x))$. By the previous observation, $\varphi_x$ is independent of the choice of $x$, and so we set $\varphi := \varphi_x$. 

The map $\varphi\colon \partial X \to \mathbb{R}_{\geq 0}$ is continuous with respect to the cone topology. Indeed, fixing some base point $x\in X$, for every $\xi\in \partial X$ and $\epsilon>0$, we know that $U(\xi,1,\frac{\epsilon}{\|f\|})\subset \varphi^{-1}(B_\epsilon(\varphi(\xi)))$, where $B_\epsilon(\varphi(\xi))$ is the open ball of radius $\epsilon$ centred at $\varphi(\xi)$. Thus $\varphi$ is continuous.

Let $\xi\in \partial X$. By the above, we know that $\varphi$ is constant on the set $\text{Ant}_{vis}(\xi)$ of points visually antipodal to $\xi$. By induction, $\varphi$ is also constant on $\text{Ant}_{vis}^j(\xi)$ for all $j\in \N$, and therefore on the minimal visual symmetric subset $\text{A}_{vis}(\xi)$. Furthermore, by continuity, $\varphi$ is constant on $B_\xi:=\overline{\text{A}_{vis}(\xi)}$. By Lemma~\ref{St2} (cf. \cite[Corollaries 3.5 and 3.6]{MR4701879}), the closure $B_\xi$ is symmetric. 

Now suppose, $B_\xi=\partial X$. Then $\varphi$ is constant on the entire boundary, which contradicts the fact that the affine map $f$ is not a dilation. Indeed take two geodesics $\eta_1\colon I_1 \to X$ and $\eta_2\colon I_2 \to X$ with intervals $I_1,I_2$. By geodesic completeness, they can be extended to geodesic lines respectively. If $\varphi$ was constant, these lines, and therefore also $\eta_1$ and $\eta_2$, would both be reparametrized under $f$ by the same factor and hence $f$ would be a dilation. 

Thus $B_\xi\subset \partial X$ is a non-trivial closed symmetric subset. By Stadler's Theorem~\ref{St} (cf. \cite[Theorem A]{Stadler}), this finally establishes that $X$ is a Riemannian symmetric space of higher rank, a Euclidean building of higher rank or spilts non-trivially as a metric product.  
\end{proof}

\begin{figure}[h]
 \begin{center}
  \begin{tikzpicture}
   \draw (0, 0) circle (2);

   \coordinate (X) at (0.5, 0.5) {};
   \coordinate (xi0) at (120: 2) {};
   \coordinate (xi1) at (20: 2) {};
   \coordinate (xi2) at (260: 2) {};

   \node[below left] at (X) {$x$};
   \node[above left] at (xi0) {$\xi$};
   \node[above right] at (xi1) {$\xi'$};
   \node[below] at (xi2) {$\xi''$};

   \draw[blue] (X) -- (xi0);
   \draw[blue] (X) -- (xi1);
   \draw[blue] (X) -- (xi2);

   \draw[blue] (xi0) to[bend right=15] (xi1);
   \draw[blue] (xi2) to[bend right=15] (xi0);

  \end{tikzpicture}
  \caption{The map $\varphi:\partial X \to \R_{\geq 0}$ is constant on the set $\text{Ant}_{vis}(\xi)$ of points visually antipodal to $\xi$.}

 \end{center}
 \end{figure}
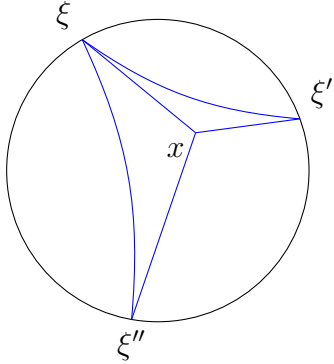

\section{Necessity of assumptions and open questions}\label{ness}
We briefly explore the necessity of the assumptions in Theorem~\ref{main}.
By Theorem~\ref{main}, a \textit{geodesically complete, locally compact CAT$(0)$ space with a geometric group action} admits an affine map which is not a dilation if and only if it is isometric to one of the three \textit{higher rank model geometries}: non-trivial metric products, Riemannian symmetric spaces of higher rank, and Euclidean buildings of higher rank. 

Our argument directly builds upon, and is a consequence of, Stadler's recent characterization (cf. \cite[Theorem A]{Stadler}) of these higher rank model geometries in terms of symmetric sets at infinity. The assumptions --- geodesic completeness, local compactness, and that the space admits a geometric group action --- stem directly from Stadler's result.

However, concerning the geometric group action, we are not aware of any counterexamples if we drop this assumption, indeed it seems reasonable to conjecture that the result still holds under these weaker assumptions and it is an interesting question to investigate this further. 

Regarding dropping the local compactness, notice that our definition of Euclidean buildings already allows for, and Lemma~\ref{buildings} already holds for Euclidean buildings that are not locally compact. In contrast, classical Riemannian symmetric spaces are, by definition, locally compact, and we would likely need to generalize to infinite dimensional Riemannian symmetric spaces. For more on this, see for example \cite{MR3449152}.

Finally, in contrast to the other assumptions, we record that geodesic completeness is clearly necessary. Take a non-trivial metric product of Hadamard spaces $X\times Y$, or a Riemannian symmetric space $M$ of higher rank (and non-compact type). As we already noted in the proof of Theorem~\ref{main}, both of these types of spaces admit affine maps that are not dilations. A closed ball in any of these spaces is convex and CAT$(0)$; and its compactness ensures that it trivially admits a geometric group action. An affine map that is not a dilation on the entire space remains a non-dilation when restricted to a closed ball. However, balls are clearly not isometric to non-trivial products, Riemannian symmetric spaces, or Euclidean buildings. 

Indeed, it might be hoped that non-dilation affine maps on any CAT$(0)$ space arise precisely as above: as non-dilation affine maps on one of the three higher rank model geometries, restricted to convex subsets. While this holds in the special case of real valued affine maps as shown in \cite{MR2262730}, the general case currently appears insurmountable and new ideas are likely needed.

\section{Acknowledgments} Foremost, I would like to thank Alexander Lytchak for helpful comments and suggestions. I also thank Stephan Stadler for comments and answering questions on his work, Julia Heller and Sven-Ole Behrend for discussions on buildings, and Jo\~{a}o Lobo Fernandes for helpful comments. Finally, I thank Johannes Gigla for assistance in creating the graphics. This
research was partially supported by the Deutsche Forschungsgemeinschaft
(DFG, German Research Foundation) under project number 281869850.

\bibliographystyle{alpha}
\bibliography{paper}

\end{document}